\documentclass[11pt]{article}
\usepackage[pass]{geometry}

\usepackage[utf8]{inputenc}

\usepackage{amsmath, amssymb, amsthm}
\usepackage[abbrev,alphabetic]{amsrefs}

\BibSpec{misc}{%
    +{}  {\PrintAuthors}                {author}
    +{,} { \textit}                     {title}
    +{.} { }                            {part}
    +{:} { \textit}                     {subtitle}
    +{,} { \PrintContributions}         {contribution}
    +{.} { \PrintPartials}              {partial}
    +{,} { }                            {journal}
    +{}  { \textbf}                     {volume}
    +{}  { \PrintDatePV}                {date}
    +{,} { \issuetext}                  {number}
    +{,} { \eprintpages}                {pages}
    +{,} { }                            {status}
    +{,} { \url}                        {url}
    +{,} { \PrintDOI}                   {doi}
    +{,} { available at \eprint}        {eprint}
    +{}  { \parenthesize}               {language}
    +{}  { \PrintTranslation}           {translation}
    +{;} { \PrintReprint}               {reprint}
    +{.} { }                            {note}
    +{.} {}                             {transition}
    +{}  {\SentenceSpace \PrintReviews} {review}
}

\usepackage{fouriernc}

\usepackage{amsfonts}
\usepackage{mathtools}
\usepackage{microtype}
\usepackage{hyperref}
\hypersetup{
  colorlinks,
  citecolor=blue,
  }
\usepackage{fancyhdr}
\usepackage[capitalize]{cleveref}

\usepackage[dvipsnames]{xcolor}
\usepackage{multirow}
\usepackage[scale=0.8]{sourcecodepro}

\newtheorem{theorem}{Theorem}

\newcommand{\Aut}{\operatorname{Aut}}
\newcommand{\SAut}{\operatorname{SAut}}
\newcommand{\SL}{\operatorname{SL}}
\newcommand{\GL}{\operatorname{GL}}

\newcommand{\Z}{\mathbb Z}
\newcommand{\R}{\mathbb R}

\newcommand{\Sq}{\operatorname{Sq}}
\newcommand{\Adj}{\operatorname{Adj}}
\newcommand{\Op}{\operatorname{Op}}

\newcommand{\supp}{\operatorname{supp}}

\newcounter{dawidcomments}

\newcounter{marekcomments}

\newcounter{piotrcomments}

\numberwithin{figure}{section}
\numberwithin{table}{section}

\theoremstyle{plain}
\newtheorem{thm}{Theorem}[section]
\crefname{thm}{Theorem}{Theorems}
\newtheorem*{prop*}{Proposition}
\newtheorem*{thm*}{Theorem}
\newtheorem{prop}[thm]{Proposition}
\crefname{prop}{Proposition}{Propositions}
\newtheorem{lem}[thm]{Lemma}
\crefname{lem}{Lemma}{Lemmata}
\newtheorem{cor}[thm]{Corollary}
\crefname{cor}{Corollary}{Corollaries}

\crefname{conj}{Conjecture}{Conjectures}
\crefname{equation}{Equation}{Equations}

\theoremstyle{definition}

\newtheorem{dfn}[thm]{Definition}
\newtheorem*{dfn*}{Definition}

\theoremstyle{remark}
\newtheorem{rmk}[thm]{Remark}

\newtheoremstyle{maintheorem}{}{}{\itshape}{}{\bfseries}{}{.5em}{#1 \!\thmnote{#3}.}
\theoremstyle{maintheorem}

\makeatletter
\let\c@figure\c@thm
\let\c@table\c@thm
\makeatother
\crefname{figure}{Figure}{Figures}
\crefname{table}{Table}{Tables}

\newcommand{\SOSconedist}[1]{2^{2\left\lceil\log_2 #1 \right\rceil}}
\newcommand{\SOSconedistm}[2]{2^{2\left\lceil\log_2 #1\right\rceil #2}}

\newcommand{\ourtitle}{On property (T) for $\Aut(F_n)$ and $\SL_n(\Z)$}
\newcommand{\IMPAN}{Institute of Mathematics, Polish Academy of Sciences, Warsaw, Poland}
\newcommand{\UAM}{Adam Mickiewicz University, Poznań, Poland}
\newcommand{\TUB}{Technische Universität Berlin, Germany, Chair of Discrete Mathematics/Geometry}
\newcommand{\UB}{Universit\"{a}t Bielefeld, Bielefeld, Germany}

\title{\ourtitle}
\author{Marek Kaluba{\footnote{\UAM; \TUB}}, Dawid Kielak{\footnote{\UB}} and Piotr Nowak{\footnote{\IMPAN}}}
\date{\today}

\begin{document}
\pagestyle{fancy}
\maketitle
\abstract{
We prove that $\Aut(F_n)$ has Kazhdan's property (T) for every $n \geqslant 6$.
Together with a previous result of Kaluba, Nowak, and Ozawa, this gives the same statement for $n\geqslant 5$.


We also provide explicit lower bounds for the Kazhdan constants of $\SAut(F_n)$ (with $n \geqslant 6$) and of $\SL_n(\Z)$ (with $n \geqslant 3$) with respect to natural generating sets.
In the latter case, these bounds improve upon previously known lower bounds whenever $n > 6$.
}

\section{Introduction}

The study of $\Aut(F_n)$, the group of automorphisms of the $n$-generated free group, is motivated to a large extent by the apparent similarities between $\Aut(F_n)$ and $\SL_n(\Z)$, the special linear group over the integers.
The latter is a lattice in a semi-simple Lie group (of higher rank, if we assume that $n \geqslant 3$), which has far reaching consequences.
In particular, linear representation theory of $\SL_n(\Z)$ is very rigid -- this is reflected by Margulis's Superrigidity and by Kazhdan's property~(T).

The question of whether $\Aut(F_n)$
has Kazhdan's property~(T) is very natural in the context of the search for similarities between $\Aut(F_n)$ and $\SL_n(\Z)$.
It was raised by many authors, e.g.,
\cite{LubotzkyPak2001}, \cite{BridsonVogtmann2006}*{Question 7}, \cite{BogopolskiVikentiev2010}*{page 4}, and \cite{Breuillard2014}*{page 345}.

For $n=2$, the group $\Aut(F_2)$ maps onto $\mathrm{Out}(F_2) \cong \GL_2(\Z)$, a virtually free group, and thus does not have property~(T).
For $n=3$, the group $\Aut(F_3)$ maps onto $\mathrm{Out}(F_3)$, which is a virtually residually torsion-free nilpotent group, as shown by J.~McCool~\cite{McCool1989}, and hence does not have property~(T).
It was later shown by F.~Grunewald and A.~Lubotzky in \cite{GrunewaldLubotzky2006} that $\Aut(F_3)$ is large, that is, that it virtually maps onto $F_2$.
This fact of course also implies that $\Aut(F_3)$ does not have property~(T).
For $n=4$, the problem remains open.
The first positive result has been obtained by N.~Ozawa and the first and third author in \cite{Kalubaetal2017}, where it was shown that $\Aut(F_5)$ does have property~(T).

\smallskip
Property (T) admits many equivalent definitions;
the one which is central to our work is due to Ozawa \cite{Ozawa2016} and is expressed in terms of the Laplace element in $\R G$.
Let $G$ be a group with a finite symmetric generating set $S$.
The Laplacian $\Delta \in \R G$ is defined by
\begin{equation}
 \Delta = \vert S \vert -\sum_{s \in S} s = \frac{1}{2}\sum_{s\in S}(1-s)^*(1-s),\label{eq:theLaplacian}
\end{equation}
where $\ast \colon \R G \to \R G$ is induced by $g \mapsto g^{-1}$.
The group $G$ is said to satisfy \emph{Kazhdan's property~(T)} if there exists $\lambda > 0$ and finitely many elements $\xi_i \in \R G$ such that
\begin{equation}\label{Taka's equation}
 \Delta^2 - \lambda \Delta = \sum_i \xi_i^* \xi_i.
\end{equation}

It can be shown that property~(T) is independent of the finite generating set chosen.
For a fixed $S$ however, it is interesting to ask:
how much of the group do we need to explore to find such elements $\xi_i$ (should they exist)?
More formally, we say that the pair $(G,S)$ has \emph{Kazhdan radius} at most $R$ if
(for some $\lambda > 0$) the element $\Delta^2 - \lambda \Delta$ admits a sum of squares decomposition as above,
with the elements $\xi_i$ supported in the ball of radius $R$ around $1$ in the word-length metric on $(G, S)$.

\medskip
The main result of this paper is the following (see also \cref{main thm long}).

\begin{theorem}\label{main theorem}
$\Aut(F_n)$ has property~(T) for all $n\geqslant 6$.
\end{theorem}
\noindent Together with \cite{Kalubaetal2017} we obtain the immediate corollary that $\Aut(F_n)$ has property~(T) for all $n \geqslant 5$.
It is worth noting that the results and methods of \cite{Kalubaetal2017} are instrumental for the results presented here.

\medskip
The proof of \cref{main theorem} is given in \cref{section (T) for AutFn}.
For technical reasons, we work with the group $\SAut(F_n)$, which is the kernel of the determinant map $\Aut(F_n) \to \GL_n(\Z) \to \Z / 2\Z$ and a subgroup of index 2 in $\Aut(F_n)$.
For $\SAut(F_n)$ we provide an explicit estimate on Kazhdan constants and we also show that the Kazhdan radius for $\SAut(F_n)$ is at most 2 for all $n\geqslant 5$.
Similarly, we give a new proof of property~(T) for $\SL_n(\Z)$ for $n \geqslant 3$, and bound the Kazhdan radius above by $2$ as well.
This aligns with the numerical results of \cite{KalubaNowak2018} where a decomposition as in \eqref{Taka's equation} of radius $2$ has been found for $n\leqslant 6$.
The new argument, however, provides necessarily weaker bounds on the Kazhdan constants than \cite{KalubaNowak2018} (for $n = 3,4,5$) -- see \cref{rmk: Explicit Kazhdan constants}.

In both cases of groups, we work with natural generating sets: for $\SL_n(\Z)$ we take $S$ to be the set of all elementary matrices ${E_{ij}}^{\pm 1}$, where $E_{ij}$ denotes the matrix which differs from the identity only in position $(i,j)$, where its entry is $1$.
For $\SAut(F_n)$, we take $S$ to consist of right and left Nielsen transvections and their inverses. More explicitly, we have
\[
 S = \{ {\rho_{ij}}^{\pm 1}, {\lambda_{ij}}^{\pm 1}  \mid 1 \leqslant i,j \leqslant n, i \neq j \},
\]
where, denoting a fixed basis of $F_n$ by $\{a_1, \dots, a_n\}$,
\[
 \rho_{ij}(a_k) = \left\{ \begin{array}{lcl}
                           a_k & \multirow{2}{*}{\text{if}} & k \neq i, \\
                           a_i a_j & & k=i,
                          \end{array} \right.
\ \ \
\lambda_{ij}(a_k) = \left\{ \begin{array}{lcl}
                           a_k & \multirow{2}{*}{\text{if}} & k \neq i, \\
                           a_j a_i & & k=i.
                          \end{array} \right.
\]

\smallskip

To prove our results we decompose the left hand side of the equation
\eqref{Taka's equation} into summands that behave well under symmetrisation by the alternating groups acting naturally on the bases elements of $F_m$ and $\Z^m$.
In particular, if such summands for a given $m$ admit sum of squares decomposition as in \eqref{Taka's equation}, so do their symmetrisations for all $n>m$.
Using this method we  reduce the problem of determining property~(T) for an infinite family of groups to finding a single sum of squares decomposition of a single element in a group algebra. 

As explained in \cite{KalubaNowak2018}*{Section~2}, finding such a decomposition can be accomplished by producing a single real positive definite matrix.
This positive definite matrix is then found numerically, using a solver designed to perform semidefinite optimization.
An additional argument following \cite{Ozawa2016} and \cite{Netzer2015}, that relies on the order structure of the self-adjoint elements in the augmentation ideal, allows for the turning of the numerical result into a proof of the existence of an exact solution, i.e. leads to a rigorous proof of property~(T).

Theorem \ref{main theorem} has several applications: it explains the fast convergence of the product replacement algorithm \cite{LubotzkyPak2001}; it implies Conjecture 4.1 for $n\ge5$ in \cite{Fisher-IMRN} as a special case; finally, it provides explicit generating sets for sequences of alternating groups that turn them into sequences of expanders~\cite{Gilman1977}.

Computational methods for property~(T) have been used successfully by T.~Netzer and A.~Thom \cite{Netzer2015}, K.~Fujiwara and Y.~Kabaya in \cite{Fujiwara2017}  M.~Kaluba and P.~Nowak in \cite{KalubaNowak2018} and the two authors and N.~Ozawa in \cite{Kalubaetal2017}.
However, they concerned specific examples only: $\SL_n(\mathbb{R})$ for $n=3,4,5,6$, various related finitely presented groups, and $\SAut(F_5)$.

\subsubsection*{Acknowledgements}

The first author was supported
by the National Science Center, Poland grant 2017/26/D/ST1/00103 \emph{Computational aspects of property~(T)} and by Deutsche Forschungsgemeinschaft (EXC 2046: {`MATH$^+$'}, Project EF1-3).

The second author was supported by the grant KI 1853/3-1 within the Priority Programme 2026 \href{https://www.spp2026.de/}{`Geometry at infinity'} of the German Science Foundation (DFG).

The third author was supported by the European Research Council (ERC) grant \emph{Rigidity of groups and higher index theory} under the European Union's Horizon 2020 research and innovation program (grant agreement no. 677120-INDEX).

The first and second author would like to express their gratitude to the Mathematical Institute of the Polish Academy of Sciences (IMPAN) in Warsaw, where some of this work was conducted.

\section{The Kazhdan radius}

Throughout the paper we will use the notation $B_R(G, S)$ to denote the ball of radius $R$ in $G$, centered at the identity, with  respect to the word-length metric induced by the generating set $S$.
We will often shorten this to just $B_R$ when both $G$ and $S$ could be inferred from context.

\paragraph{Cones of sums of squares}

Let $G$ be a group with a generating set $S$. We say that an element $x$ in $\R G$ admits a \emph{decomposition into sum of squares} if there exist finitely many elements $\xi_i \in \R G$ such that
\[
  x = \sum_i \xi_i^* \xi_i.
\]

\begin{dfn} \label{def: sos&radius}
  The elements of $\R G$ admitting a decomposition into sum of squares span a \emph{(positive) cone of squares in $\R G$} which we will denote by $\Sigma^2\R G$.

  Of special interest to us will be $IG$, the augmentation ideal of $\R G$, i.e. the kernel of the map $\R G \to \R$ sending $\Sigma_g c_g g \mapsto \Sigma_g c_g$. By $\Sigma^2 IG$ we denote the intersection $I G \cap \Sigma^2\R G$. Note that $\Sigma^2 IG$ coincides with the cone of elements which admit a decomposition into sum of squares of elements from $IG$.

The cones $\Sigma^2 \R G$ and $\Sigma^2 I G$ admit  natural \emph{filtrations by radii of supports}. For any ideal $I\subseteq \R G$ we have
    \[{\R}\cap I = \Sigma^2_0 I \subseteq \Sigma^2_1 I \subseteq \Sigma^2_2 I \subseteq \cdots \subseteq \Sigma^2 I,\]
  that is, we allow in $\Sigma^2_R I$ only those elements of $\Sigma^2 I$ which admit a decomposition into sum of squares of elements supported in $B_R$.

\end{dfn}


\begin{dfn}
Let $G$ be a group with a finite symmetric generating set $S$. We say that $G$ has \emph{Kazhdan radius} (with respect to $S$) at most $R$ if
  \[\Delta^2 - \lambda \Delta \in \Sigma^2_R\R G\]
for some $\lambda > 0$.
\end{dfn}
Note that we could replace $\Sigma^2_R\R G$ by $\Sigma^2_R I G$ in the definition above, since $\Delta^2 - \lambda \Delta \in I G$.

The significance of the Kazhdan radius lies in the fact that it tells us how much of the group is needed to prove the presence of property~(T).
Suppose that we have a map $j\colon B_{2R}(G,S) \to B_{2R}(G', S')$
and that its linear extension $j\colon \langle B_{2R}(G,S)\rangle_\R \to \langle B_{2R}(G', S')\rangle$
carries $\Delta_G$ to $\Delta_{G'}$, ${\Delta_G}^2$ to ${\Delta_{G'}}^2$, and $j\left(\Sigma^2_R \R G\right)\subset \Sigma^2_R \R G'$
(e.g. when $j$ is a local isomorphism).
If $G$ has property~(T) witnessed by sum of squares decomposition $\Delta_G^2 - \lambda \Delta_G = \sum_i\xi_i^*\xi_i \in \Sigma^2_R \R G$, then
\[{\Delta_{G'}}^2 -\lambda \Delta_{G'} = j({\Delta_G}^2 - \lambda \Delta_G) = \sum j\left(\xi_i^*\xi_i\right)\in \Sigma_R^2 \R G'\]
is a sum of squares as well.

This can be clearly envisioned in the case of modular projection $\rho_p \colon G = \SL_n(\Z) \to \SL_n(\mathbb{F}_p)=G'$.
For $p$ large enough the projection carries $\Delta_G$ to $\Delta_{G'}$ and ${\Delta}^2_{G}$ to ${\Delta_{G'}}^2$, and (since its an actual homomorphism) sum of squares to sum of squares.
On the other hand, a section of the modular projection is not a (global) homomorphism anymore,
but still can be used to lift sum of squares decomposition in $\SL_n(\mathbb{F}_p)$ to one in $\SL_n(\Z)$,
provided the support of squared elements is small enough.

Setting $n=2$ and observing that $\SL_2(\Z)$ does not have property (T), we conclude that the Kazhdan radius of $\SL_2(\mathbb{F}_p)$ with respect to $\rho_p(S)$ tends to infinity with $p$ for every fixed finite symmetric generating set $S$ of $\SL_2(\Z)$.

\section{Squaring the Laplacian}

In this section we give a new, algebraic method of proving property~(T) in $\SL_n(\Z)$ and $\SAut(F_n)$.
Most of the proofs are identical for both types of groups. Therefore, we will present them in a uniform way -- to do that, we introduce some notation.

\paragraph{The setting} We identify $\Z^n$ and $F_n$ with the free-abelian and free, respectively, groups generated by $N_n = \{a_1, \dots, a_n \}$.
Extending $N_{n+1} = N_n \cup \{a_{n+1}\}$ induces natural inclusions $\Z^n \to \Z^{n+1}$ and $F_n \to F_{n+1}$.
In the following, $G_n$ will be either $\SL_n(\Z)$ or $\SAut(F_n)$.
The groups $G_n$ come with inclusions $G_n \to G_{n+1}$ induced by the inclusions $N_n \to N_{n+1}$ in the obvious way.
We also fix a finite symmetric generating set $S_n$ of $G_n$:
\begin{itemize}
\item for $\SL_n(\Z)$, we take all elementary matrices ${E_{ij}}^{\pm 1}$;
\item for $\SAut(F_n)$, we take all Nielsen transvections and their inverses, namely ${\rho_{ij}}^{\pm 1}$ and ${\lambda_{ij}}^{\pm 1}$.
\end{itemize}
It is clear that the inclusion $G_n \to G_{n+1}$ takes $S_n$ into $S_{n+1}$.

The action of $A_n$, the alternating group of rank $n$, on $N_n$ induces an action of $A_n$ on $G_n$ by automorphisms which preserves the generating set $S_n$.
We will denote the action by $\sigma(g)$ with $\sigma \in A_n$ and $g \in G_n$.
The action of $A_{n+1}$ on $G_{n+1}$ restricts to the action of $A_n$ on the embedded copy of $G_n$,
that is, the inclusion $G_n \to G_{n+1}$ is $A_n$-equivariant.

Let $C_n$ denote the $(n-1)$-simplex with the vertex set $\{1, \dots, n \}$. The group $A_n$ acts on $C_n$ by permuting the vertices.
The set of edges (unordered pairs $\{i,j\}$ of distinct integers) $E_n$ of $C_n$ inherits the action of $A_n$, namely
\[
 \sigma(e) = \sigma\big( \{i,j\}\big) = \big\{ \sigma(i), \sigma(j) \big\}.
\]

Since generators in the set $S_n$ are double-indexed, there is an obvious function $l_n \colon S_n \to E_n$ which sends a generator to the edge between vertices labeled after its indices:
for $\SL_n(\Z)$, the function sends ${E_{ij}}^{\pm 1}$ to $\{i,j\}$;
for $\SAut(F_n)$, it sends ${\rho_{ij}}^{\pm 1}$ and ${\lambda_{ij}}^{\pm 1}$ to $\{i,j\}$.
It is immediate that $l_{n+1}\vert_{S_n} = l_n$.

Given two edges $e,f \in E_n$, we have the following trichotomy:
\begin{enumerate}
  \item the edges may coincide;
  \item they may be \emph{adjacent}, that is overlap in a single vertex;
  \item or they may be \emph{opposite}, that is disjoint.
\end{enumerate}
The function $l_n$ has the following property:
given two elements $s, t \in S_n$, if $l_n(s)$ and $l_n(t)$ are opposite edges, then $s$ and $t$ commute in $G_n$.

\smallskip
To summarise, we have the following:
\begin{itemize}
 \item a family of groups $G_n$ with finite symmetric generating sets $S_n$,
 \item the alternating group $A_n$ of rank $n$ acting on $G_n$ by automorphisms  preserving $S_n$,
 \item an $A_n$-equivariant inclusion $G_n \leqslant G_{n+1}$,
 \item an $A_n$-equivariant map $l_n \colon S_n \to E_n$ onto the set of edges of the standard $n-1$ simplex $C_n$, such that if $l_n(s)$ and $l_n(t)$ are opposite edges, then $s$ and $t$ commute in $G_n$.
\end{itemize}

Since we have inclusions $G_n \leqslant G_{n+1}$, we also have $\R G_n \subseteq \R G_{n+1}$.
The action of $A_n$ on $G_n$ induces an action on $\R G_n$ by linear extension;
it is clear that this action and every other action on $G_n$ by automorphisms preserves $\Sigma^2\R G_n$, the positive cone of sums of squares, set-wise.
More importantly: since the action of $A_n$ permutes the elements of the generating set $S_n$ it does not alter the word-length, and hence it preserves also the filtration of $\Sigma^2 \R G$ by radii.


\paragraph{The Laplacians}

\begin{dfn}
For $n \geqslant 2$, we define $\Delta_n$ to be the (non-normalised) Laplacian of $G_n$, as defined in \eqref{eq:theLaplacian}, with respect to the generating set $S_n$.

For $e = \{i,j\}  \in E_n$, we define
 $S_e = \{ t \in S_n \mid l_n(t) = e \}$ and
 \[
  \Delta_e = \vert S_e \vert - \sum_{t \in S_e} t
 \]
to be the Laplacian of the group $\left\langle S_e \right\rangle \leqslant G_n$ with respect to the generating set $S_e$.
\end{dfn}
In particular, $\Delta_{\{1,2\}}$ coincides with the Laplacian $\Delta_2$.
In fact, we have more: for every $\sigma \in A_n$ we have
\[
 \sigma\left(\Delta_e\right) = \Delta_{\sigma(e)},
\]
and so every $\Delta_e$ is equal to $\sigma(\Delta_2)$ for some $\sigma \in A_n$. This $\sigma$ is not unique -- what is unique is the coset of $\mathrm {Stab}(\{1,2\})$ in $A_n$ defined by $\sigma$.

\begin{lem}
 \label{delta using edges}
 For every $n \geqslant 3$ we have
 \[
  \Delta_n = \sum_{e \in E_n} \Delta_e = \frac{1}{(n-2)!}\sum_{\sigma\in A_n } \sigma(\Delta_2).
 \]
\end{lem}
\begin{proof}
 This follows directly from the definitions:
 \begin{align*}
    \Delta_n &=
        \vert S_n \vert - \sum_{t \in S_n} t \\
        &= \sum_{e \in E_n} \vert S_e \vert - \sum_{e \in E_n} \sum_{l_n(t) = e} t \\
        &= \sum_{e \in E_n} \big( \vert S_e \vert - \sum_{l_n(t) = e} t \big) \\
        &= \sum_{e \in E_n} \Delta_e.
 \end{align*}
Since the $A_n$-stabiliser of $e$ is isomorphic to the symmetric group of rank $n-2$, the second formula holds as well.
\end{proof}

\begin{cor}
\label{delta comp}
For $m \geqslant  n \geqslant 3$ we have
\[
\sum_{\sigma \in A_m} \sigma(\Delta_n) =
\frac {\vert A_n \vert }{(n-2)!} \sum_{\sigma \in A_m} \sigma(\Delta_2) =
\frac {\vert A_n \vert \cdot (m-2)!}{  (n-2)! } \Delta_m =
\binom n 2 \cdot (m-2)!  \Delta_m.
\]
\end{cor}

\subsubsection*{Decomposing $\Delta^2$}

For an edge $e \in E_n$ in the standard simplex we define $\Adj(e)$ to be the set of edges in $E_n$ adjacent to $e$,
and $\Op(e)$ to be the set of edges opposite to $e$.

\begin{dfn}
For $n \geqslant 3$ we define $\Sq_n, \Adj_n, \Op_n \in \R G_n$ by
\begin{align*}
\Sq_n &=
    \sum_{e \in E_n} {\Delta_{e}}^2,\\
\Adj_n &=
    \sum_{e \in E_n} \Big( \Delta_{e} \sum_{f \in \Adj(e)} \Delta_f \Big),\\
\Op_n &=
    \sum_{e \in E_n} \big( \Delta_{e} \sum_{f \in \Op(e)} \Delta_f \big).
\end{align*}
Our convention is that empty sums are equal to $0$, and so $\Op_3 = 0$.
\end{dfn}

It is not hard to check that  $\Sq_n, \Adj_n$, and $\Op_n$ can be expressed in terms of the $A_n$-action as well.
Fix $e = \{1,2\}$ (or any other edge in $E_n$).
Then
\begin{align*}
\Sq_n  &=
    \frac{1}{(n-2)!} \sum_{\sigma \in A_{n}} \sigma({\Delta_e}^2),\\
\Adj_n &=
    \frac{1}{(n-2)!^2} \sum_{\sigma \in A_{n}} \sigma(\Delta_e ) \left( \sum_{\tau(e) \in \Adj(\sigma(e))} \tau(\Delta_e)\right),\\
\Op_n &=
    \frac{1}{(n-2)!^2} \sum_{\sigma \in A_{n}} \sigma(\Delta_e ) \left( \sum_{\tau(e) \in \Op(\sigma(e))} \tau(\Delta_e)\right).
\end{align*}
From these expressions it is clear that $\Sq_n$, $\Adj_n$ and $\Op_n$ belong to the augmentation ideal $IG_n$ of $\R G_n$ and are $*$-invariant and $A_n$-invariant.

\begin{lem}
\label{sum of sq adj op}
For $n \geqslant 3$ we have
\[
 \Sq_n + \Adj_n +\Op_n = {\Delta_n}^2 \in \Sigma^2_1 \R G.
\]
\end{lem}
\begin{proof}
 We have
 \begin{align*}
 \Sq_n + \Adj_n +\Op_n &= \sum_{e \in E_n} {\Delta_e}^2 +  \sum_{e \in E_n} \big( {\Delta_e}  \sum_{f \in \Adj(e)} {\Delta_f} \big) +  \sum_{e \in E_n} \big( {\Delta_e}  \sum_{f \in \Op(e)} {\Delta_f} \big) \\
 &=   \sum_{e \in E_n} \Big( {\Delta_e}  \big( \Delta_e + \sum_{f \in \Adj(e)} {\Delta_f} + \sum_{f \in \Op(e)} {\Delta_f} \big) \Big) \\
  &= \sum_{e \in E_n}  {\Delta_e} \Delta_n \\
   &= \left( \sum_{e \in E_n}  {\Delta_e} \right) \Delta_n \\
   &= {\Delta_n}^2
 \end{align*}
Since $\Delta_n$ is $*$-invariant and supported on the ball of radius $1$, the result follows.
\end{proof}

\begin{lem}
\label{sq op sos}
 We have $\Sq_n \in \Sigma_1^2 \R G_n$ and $\Op_n \in \Sigma_2^2 \R G_n$.
\end{lem}
\begin{proof}
 The element $\Sq_n$ is a sum of squares of Laplacians, which are supported in $B_1$.

For $n\leqslant3$ we have $\Op_n = 0$ by definition. For $n>3$, the element $\Op_n$ is a sum of elements of the form $\Delta_e \Delta_f$ with $f \in \Op(e)$, hence for every $t \in S_e$ and $s \in S_f$, the elements $t$ and $s$ commute in $G_n$.
Now
\begin{align*}
 \Delta_e \Delta_f &= \big( \vert S_e\vert - \sum_{t \in S_e} t \big) \big( \vert S_f\vert - \sum_{s \in S_f} s \big) \\
 &= \frac 1 4 \big( \sum_{t \in S_e} (1-t^{-1})(1-t) \big) \big( \sum_{s \in S_f} (1-s^{-1})(1-s) \big) \\
 &= \frac 1 4 \sum_{(t,s) \in S_e \times S_f} \big( (1-t)(1-s) \big)^* \big((1-t)(1-s)\big)
\end{align*}
which proves the claim.
\end{proof}

\begin{lem}
\label{Adj comp}
For $m \geqslant n \geqslant 3$ we have
\[
\sum_{\sigma \in A_m} \sigma(\Adj_n) = n(n-1)(n-2) \frac {(m-3)!}{2} \Adj_m.
\]
\end{lem}
\begin{proof}
This is a direct computation.
If $m=n$ then the result follows immediately, since then $n(n-1)(n-2) \frac {(m-3)!} 2 = \vert A_n \vert$ and $\Adj_n$ is $A_n$\nobreakdash-\hspace{0pt}invariant.

Suppose that $m \neq n$.
The case of $m=4$ (and hence $n=3$) is special, since $A_4$ does not act transitively on pairs $(e,f)$ with $e \in E_3$ and $f \in \Adj(e)$;
instead, the action has two orbits, each of cardinality $3$.
Also, the stabiliser in $A_4$ of a pair $(e,f)$ is the same as the stabiliser of $3$ points in the natural action of $A_4$ on $4$ points, i.e. is trivial.
It is clear that the coefficient next to $\Adj_m$ on the right hand side is equal to the product of the cardinality of this stabiliser and any of the orbits.
We thus have
\[
\sum_{\sigma \in A_4} \sigma(\Adj_3) = 3 \Adj_4
\]
which agrees with the statement.

Now suppose that $m > 4$.
The action of $A_m$ on the pairs $(e,f)$ is transitive,
there are $ n(n-1)(n-2)$ such pairs,
and the stabiliser has cardinality $\frac {(m-3)!} 2$.
Therefore
\[
\sum_{\sigma \in A_m} \sigma(\Adj_n) = n(n-1)(n-2) \frac {(m-3)!}{2} \Adj_m.\qedhere
\]
\end{proof}

\begin{lem}
\label{Op comp}
For $m \geqslant n \geqslant 4$ we have
\[
\sum_{\sigma \in A_m} \sigma(\Op_n) = 2\binom{n}{2} \binom{n-2}{2} {(m-4)!} \Op_m.
\]
\end{lem}
\begin{proof}
This is  a simple computation as before: there are $\binom{n}{2}\binom {n-2}{2}$ pairs $(e,f)$ with $e \in E_n$ and $f \in \Op(e)$,
the action of $A_m$ is transitive on such pairs,
and the stabiliser of a single such pair has cardinality $4 \vert A_{m-4} \vert = 2 \cdot {(m-4)!}$.
\end{proof}

\section{Proving property (T)}

We are now ready to give two methods of proving property~(T) for the groups $G_n$.

\subsection{The Methods}

\subsubsection*{Method I: using $\Adj_n + k\Op_n$}
This method will be applied to $\SL_n(\Z)$ for $n \geqslant 3$ and $\SAut(F_n)$ for $n \geqslant 7$.

\begin{prop}[Method I]
\label{Adj+kOp>=0}
\label{method I}
Let $n\geqslant 3$ and suppose that
\[\Adj_n +k \Op_n - \lambda \Delta_n \in \Sigma^2_R \R G_n\]
for some $R \geqslant 2,k \geqslant 0$, and $\lambda > 0$.
Then $G_m$ has Kazhdan's property~(T) for every $m\geqslant n$ such that
$k(n-3) \leqslant m-3$.
Moreover, the Kazhdan radius of $(G_m,S_m)$ is bounded above by $R$.
\end{prop}
\begin{proof}
Suppose first that $n \geqslant 4$.
By \cref{delta comp,Adj comp,Op comp} we have
\begin{multline*}
\sum_{\sigma \in A_m} \sigma\left(\Adj_n + k \Op_n - \lambda \Delta_n\right)=
n(n-1)(n-2)\frac{(m-3)!}{2} \Adj_m + \\
+ 2k \binom{n}{2}\binom{n-2}{2}(m-4)!\Op_m -
\lambda \binom{n}{2}(m-2)!\Delta_m\\
= \frac{n(n-1)(n-2)}{2}(m-3)!\left(\Adj_m + \frac{k(n-3)}{m-3}\Op_m - \frac{\lambda(m-2)}{n-2}\Delta_m\right).
\end{multline*}
Moreover, since $A_m$ acts on $G_m$ by automorphisms, the action on $\Sigma^2\R G$ preserves filtration by radii of support, hence the element belongs to $\Sigma^2_R \R G_m$. Using the equality above and \cref{sum of sq adj op} we obtain
\begin{align*}
 {\Delta_m}^2 - \frac{\lambda(m-2)}{n-2} \Delta_m
 & = \Sq_m + \Adj_m + \Op_m - \frac{\lambda(m-2)}{n-2}  \Delta_m\\
 & = \Sq_m + \left(1-\frac{k(n-3)}{m-3}\right)\Op_m +\\
 & \quad + \frac{2}{n(n-1)(n-2)(m-3)!}\sum_{\sigma \in A_m}\sigma \left(\Adj_n + k \Op_n - \lambda \Delta_n \right).
\end{align*}
By \cref{sq op sos}, both the elements $\Sq_m$ and $\Op_m$ belong to $\Sigma^2_2 \R G_m$.
Therefore whenever $1-\frac{k(n-3)}{m-3} \geqslant 0$ we have
\[{\Delta_m}^2 - \frac{\lambda(m-2)}{n-2} \Delta_m \in \Sigma^2 _R \R G_m,\] which is equivalent to property~(T) for $G_m$ by \cite{Ozawa2016}.

The case $n = 3$ is special, since $\Op_3 = 0$ and then
\[\Delta_m^2 - (m-2)\lambda \Delta = \Sq_m + \Op_m + \frac{1}{3\cdot (m-3)!}\sum_{\sigma\in A_m} \sigma\left(\Adj_3 -\lambda\Delta_3\right) \in \Sigma^2_R \R G_m.\]
\end{proof}

The following corollary is a direct consequence of e.g. \cite{Bekkaetal2008}*{Remark 5.4.7}.

\begin{cor}
Under the same assumptions as above,
the spectral gap of $\Delta_m$ is bounded below by $\frac{\lambda(m-2)}{n-2}$, which leads to a lower bound of
\[\sqrt{\frac{2\lambda(m-2)}{(n-2)|S_m|}} \leqslant \kappa(G_m, S_m)\]
on the Kazhdan constant.
\end{cor}

\subsubsection*{Method II: using $\Adj_5 + k \Op_5$ and induction}

This method will be applied to $\SAut(F_6)$.

Before proceeding, let us define
\[h_n = \frac{(n-2) \cdot (n-1)!}{2}\] which is the number of elements $\sigma \in A_n$ such that $\{1,2\} \subset \sigma(\{1,\ldots, n-1\})$. (The notation $h$ stands for `hypersimplex', since $h_n$ gives the number of elements $\sigma$ taking a fixed hypersimplex to one containing a given edge.)

\begin{lem}\label{lem: induction}
For every $n \geqslant 3$ we have
\[
 {\Delta_{n}}^2 - \lambda \Delta_{n} =
  \frac 1 {{h_{n}}^2} \left(
    \sum_{\sigma \in A_{n}} \sigma\left(
      {\Delta_{n-1}}^2 -\lambda h_{n} \Delta_{n-1}
    \right) + X_{n}
  \right),
\]
where
\[
X_n = \sum_{\sigma \in A_n} \left(
       \sigma\left( \Delta_{n-1}\right)
       \sum_{\tau \in A_n, \tau \neq \sigma} \tau\left(\Delta_{n-1}\right)
   \right).
\]
\end{lem}
\begin{proof}
By \cref{delta using edges}, we have
\begin{align*}
\frac 2 {(n-2)\cdot (n-1)!} \sum_{\sigma \in A_{n}} \sigma(\Delta_{n-1})
&= \frac 2 {(n-2)\cdot (n-3)! \cdot (n-1)!} \sum_{\sigma \in A_{n}} \sigma\big( \sum_{\tau \in A_{n-1}} \tau(\Delta_2)\big) \\
&= \frac 2 {(n-2)\cdot (n-3)! \cdot (n-1)!} \sum_{\sigma \in A_{n}} \vert A_{n-1} \vert \sigma(\Delta_2) \\
&=  \frac {2 \cdot (n-1)!} {2 \cdot (n-2)! \cdot (n-1)!} \sum_{\sigma \in A_{n}} \sigma(\Delta_2) \\
&=  \frac {1} {(n-2)! } \sum_{\sigma \in A_{n}} \sigma(\Delta_2) \\
&= \Delta_{n}.
\end{align*}
Using the definition of $h_n$ we may rewrite this as
\[
\Delta_{n} = \frac{1}{h_{n}}\sum_{\sigma \in A_{n}}\sigma(\Delta_{n-1}).
\]
Then
\begin{multline*}
{\Delta_{n}}^2 - \lambda \Delta_{n} =
  \frac{1}{{h_{n}}^2}\left(
    \left(
      \sum_{\sigma\in A_{n}}\sigma(\Delta_{n-1})
    \right)^2 - h_{n}\lambda \sum_{\sigma \in A_{n}}\sigma(\Delta_{n-1})
  \right)\\ =
  \frac{1}{{h_{n}}^2}\left(
    \left(
      \sum_{\sigma \in A_{n}} \sigma\left(
        \Delta_{n-1}^2 - h_{n}\lambda \Delta_{n-1}
      \right)
    \right) +
    \sum_{\sigma \in A_{n}} \left(
      \sigma\left(
        \Delta_{n-1}
      \right)
    \sum_{\tau \in A_{n}, \tau \neq \sigma} \tau\left(
      \Delta_{n-1}
      \right)
    \right)
  \right)\\
  = \frac 1 {{h_{n}}^2} \left(
    \sum_{\sigma \in A_{n}} \sigma\left(
      {\Delta_{n-1}}^2 -\lambda h_{n} \Delta_{n-1}
    \right) + X_{n}
  \right),
\end{multline*}
as desired.
\end{proof}

\begin{lem}
 \label{Xn}
 For $n \geqslant 4$ we have
 \[
   X_n = \alpha_n \Sq_n +
    \beta_n \Adj_n +
    \gamma_n \Op_n,
 \]
with
\begin{align*}
 \alpha_n &= h_n \left(h_n - \frac{n-2}{n-2}\right), &
 \beta_n &=  h_n \left(h_n - \frac{n-3}{n-2}\right), &
 \gamma_n &=  h_n \left(h_n - \frac{n-4}{n-2}\right).
\end{align*}
\end{lem}
\begin{proof}
Using \cref{delta using edges} we may rewrite $X_n$ as
\begin{align*}
X_n &= \sum_{\sigma \in A_n} \left(
    \sigma\left(
        \Delta_{n-1}
    \right)
    \sum_{\tau \in A_n, \tau \neq \sigma} \tau \left(
        \Delta_{n-1}
    \right)
\right) \\
& = \sum_{\sigma \in A_n} \left(
    \sigma \Big(
        {\textstyle\sum_{e\in E_{n-1}}} \Delta_e
    \Big)
    \displaystyle \sum_{\tau \in A_n, \tau \neq \sigma}
    \tau \Big(
        {\textstyle\sum_{e\in E_{n-1}}} \Delta_e
    \Big)
\right).
\end{align*}

Expanding the above product, we obtain
\[
 X_n = \sum_{e,f \in E_n} \lambda_{e,f} \Delta_e \Delta_f
\]
with $\lambda_{e,f} \in \Z$ for every $e,f \in E_n$.
Since $X_n$ is $A_n$-invariant, we have
\[
 \lambda_{e,f} = \lambda_{\sigma(e), \sigma(f)}
\]
for every $\sigma \in A_n$.
In particular, $\lambda_{e,f}$ may take one of the three potentially distinct values, namely $\alpha_n = \lambda_{e,e}$ for every $e \in E_n$, $\beta_n =\lambda_{e,f}$ for every $e \in E_n$ and $f \in \Adj(e)$, and $\gamma_n = \lambda_{e,f}$ for every $e \in E_n$ and $f \in \Op(e)$.
Therefore
\begin{align*}
 X_n &= \alpha_n \Sq_n +
    \beta_n \Adj_n +
    \gamma_n \Op_n.
\end{align*}
The coefficients $\alpha_n, \beta_n$ and $\gamma_n$ may easily be computed: Fix an edge $e \in E_n$, and denote by $H$ the hypersimplex of $C_n$ spanned by vertices $\{1, \dots, n-1 \}$.

Let us start with $\alpha_n$.  The coefficient $\alpha_n$ is equal to the number of pairs $(\sigma, \tau) \in A_n$ such that $\sigma \neq \tau$ and $e \in \sigma(H) \cap \tau(H)$.
There are precisely $h_n$ elements $\sigma$ satisfying the above, and so there are $h_n -1$ elements $\tau$ as required. Therefore we have
\[
 \alpha_n = {h_n}(h_n -1) = h_n \left(h_n - \frac{n-2}{n-2}\right)
\]
as claimed.

Fix an edge $f \in \Adj(e)$.
The coefficient $\beta_n$ is equal to the number of pairs $(\sigma, \tau) \in A_n$ such that $\sigma \neq \tau$, $e \in \sigma(H)$ and $f \in \tau(H)$.
There are $h_n$ elements $\sigma$ as required: for $\frac {h_n} {n-2}$ of them, $\sigma(H)$ does not contain $f$; for the remaining $h_n -  \frac {h_n} {n-2}$ of these,  $\sigma(H)$ does contains $f$. Now, for a $\sigma$ of the former kind, there are precisely $h_n$ possible elements $\tau$; if $\sigma$ is of the latter kind, there are $h_n - 1$ possibilities for $\tau$. Carrying out the arithmetic yields
\[
  \beta_n = \frac {h_n} {n-2} \cdot h_n + \left(h_n - \frac {h_n} {n-2}\right) (h_n -1) = h_n \left(h_n - \frac{n-3}{n-2}\right).
\]

The computation for $\gamma_n$ is entirely analogous: we fix $f \in \Op_n$ and see that
\[
 \gamma_n = \frac {2 h_n} {n-2} \cdot h_n + \left(h_n - \frac {2 h_n} {n-2}\right)(h_n - 1)  =  h_n \left(h_n - \frac{n-4}{n-2}\right). \qedhere
\]
\end{proof}

\begin{prop}[Method II]
 \label{Adj+kOp>=0 Xn}
Suppose that there exist $R \geqslant 2$, $\lambda >0$, $k\geqslant 0$, $n\geqslant k+3$, and $\mu \geqslant 0$ such that
\[
{\Delta_{n-1}}^2 - \lambda \Delta_{n-1} \in \Sigma_R^2 \R G_{n-1}
\]
and
\[\Adj_5 + k \Op_5 - \mu \Delta_5 \in \Sigma_R^2 \R G_5.\]
Then $G_{n}$ has Kazhdan's property~(T). Moreover, the Kazhdan radius of $(G_{n},S_{n})$ is bounded above by $R$.
\end{prop}
\begin{proof}
By \cref{lem: induction}, we have
\[
{\Delta_n}^2 - \frac {3\lambda + \mu} {3 h_{n}}\Delta_n =
  \frac{1}{{h_n}^2}
    \sum_{\sigma \in A_{n}} \sigma\left(
      {\Delta_{n-1}}^2 -\lambda  \Delta_{n-1}
    \right) +
  \frac{1}{{h_n}^2}  \left( X_{n}  - \frac{\mu h_n}{3} \Delta_n \right).
\]
Therefore, it suffices to show that $X_{n}  - \frac  {\mu h_n} 3 \Delta_n \in \Sigma^2_R \R G_n$.

By \cref{Xn} we have
\begin{align*}
 \frac {n-2} {h_n} X_n &= \big((n-2)h_n - n+2\big) \left(\Sq_n + \Adj_n + \Op_n\right) + \Adj_n + 2 \Op_n.
\end{align*}
The element
$\Sq_n + \Adj_n + \Op_n = {\Delta_n}^2 \in \Sigma^2_1 \R G$,
by \cref{sum of sq adj op}, so we are left with showing that
$\Adj_n + 2 \Op_n  - \frac{\mu(n-2)}{3} \Delta_n \in \Sigma^2_R\R G$.

By assumption, $\Adj_5 + k \Op_5 - \mu \Delta_5 \in \Sigma^2_R \R G_5$ and we have already computed in the proof of \cref{Adj+kOp>=0} that
\begin{align*}
\frac 1 {30 \cdot (n-3)!}
\sum_{\sigma \in A_n}
  \sigma\left(
    \Adj_5 + k \Op_5 - \mu \Delta_5
  \right) & =
\Adj_n + \frac{2k}{n-3} \Op_n - \frac {\mu(n-2)} 3 \Delta_n.
\end{align*}
Since $n \geqslant k+3$, and since $\Op_n \in \Sigma_R^2\R G_n$ by \cref{sq op sos}, we conclude that
\[
\Adj_n + 2 \Op_n - \frac {\mu(n-2)} 3 \Delta_n \in \Sigma_R^2 \R G_n
\]
which completes the proof.
\end{proof}

\begin{cor}
\label{Adj+kOp>=0 Xn Kaz}
Under the same assumptions as above,
the spectral gap of $\Delta_n$ is bounded below by $\frac {3\lambda + \mu} {3 h_{n}}$, which leads to lower bound of \[\sqrt{\frac {2(3\lambda + \mu)}{3 h_{n}|S_n|}} \leqslant \kappa(G_n, S_n)\]
for the Kazhdan constant.
\end{cor}

\subsection{Approximating sum of squares}

Let $IG^h$ denote the set of $*$-invariant elements in the augmentation ideal $IG$ of $\R G$.
To prove the  positivity of the element $\Delta^2 - \lambda \Delta$, a numerical method for finding a decomposition into a sum of squares in $\R G$ has been applied by \cites{Netzer2015, Fujiwara2017, KalubaNowak2018} and further developed in \cite{Kalubaetal2017}.
In fact, the method provides a proof of the existence of a sum of squares decomposition of the element.
However, a similar method can be employed to prove the existence of such decompositions for other elements in  $IG^h$.

\smallskip

We use the notation $a \leqslant_R b$ to denote that $b-a \in \Sigma^2_R IG$.

\begin{lem}
Let $a \in G$ be of word-length bounded above by $R > 1$. Then
\[(2 - a^* - a) \leqslant_{\lceil R/2\rceil} \varepsilon\Delta,\]
 for every $\varepsilon \geqslant \SOSconedist{R}$,
where $\lceil \cdot \rceil$ denotes the ceiling function.
\end{lem}
\begin{proof}

The proof is an induction on the word-length of $a$.

The case $R=1$ is needed in subsequent steps of the induction, and although not covered by the statement the conclusion holds with a larger $\varepsilon$. Let $a=s$ be a generator of $G$. Then
\[(2-s^* - s) = (1-s)^*(1-s) \leqslant_1 2\Delta.\]
If $a = st$ is a product of two generators, then
\[(2 - (st)^* - st) =
  2\left((1-s)^*(1-s)+ (1-t)^*(1-t)\right) -
  (2 - s^* - t)^*(2 - s^* - t),
\]
hence $(2 - (st)^* - st) \leqslant_1 4\Delta$.
For every $\varepsilon\geqslant4$ we have
\[\varepsilon\Delta - (2 - (st)^* - st) = (\varepsilon-4)\Delta + \left(4\Delta - (2 - (st)^* - st)\right)\]
which admits a decomposition into sum of squares supported in $B_1$.

Suppose that $a=uw$ is of word-length $R$ and each $u,w\in G$ is of word-length at most $\lceil R/2\rceil$.
In the following we use the notation $U = (1-u)$ and $W = (1-w)$. We have
\begin{align*}
(2-a^* - a) & = (1-uw)^*(1-uw)\\
& = (U + uW)^*(U+uW)\\
& = U^*U + W^*u^*uW + U^*uW + W^*u^*U\\
&= U^*U + W^*W - UW - W^*U^*,
\intertext{and therefore}
2\big(U^*U + W^*W\big)-(2-a^* - a)
& = U^*U + W^*W + UW + W^*U^*\\
& = UU^* +UW + W^*U^* + W^*W\\
& = (U^*+W)^*(U^*+W) \\
& = (2-u^*-w)^*(2-u^*-w).
\end{align*}

By the inductive step, for $\varepsilon_0 = \SOSconedist{\lceil R/2 \rceil }$ we have both $(2-u^*-u)$ and $(2-w^*-w)\leqslant_{\lceil R/4 \rceil} \varepsilon_0\Delta$, i.e.
\[
  \varepsilon_0 \Delta -U^*U = \sum_i\mu_i^*\mu_i
  \quad \text{and} \quad
  \varepsilon_0 \Delta - W^*W = \sum_i\nu_i^*\nu_i,
\]
where each $\mu_i$ and $\nu_i$ is supported on the ball of radius $\lceil R/4\rceil$.
Moreover
\begin{multline*}
4\varepsilon_0 \Delta +
2\left((U^*U - \varepsilon_0 \Delta) +
(W^*W - \varepsilon_0 \Delta)\right) - (2-(uw)^* - uw) =\\
(2-u^* - w)^*(2-u^* -w).
\end{multline*}
We have
$\SOSconedist{R}= 4 \cdot \SOSconedist{\lceil R/2 \rceil}$. Indeed, it is easy to check that the equality holds when $R$ is a power of $2$.
If $R$ is not, then
\begin{align*}
\SOSconedist{R}= \SOSconedist{(R+1)} \geqslant \SOSconedist{(2 \lceil R/2 \rceil)} = 4 \cdot \SOSconedist{\lceil R/2 \rceil}
\end{align*}
and
\begin{align*}
4 \cdot \SOSconedist{\lceil R/2 \rceil} = \SOSconedist{(2 \lceil R/2 \rceil)} \geqslant  \SOSconedist{R}.
\end{align*}

Now, for every $\varepsilon \geqslant \SOSconedist{R}=4\varepsilon_0$ we obtain
\begin{align*}
\varepsilon \Delta - (2-a^* - a)
& = (2-u^*-w)^*(2-u^*-w) \\
& \quad + 2\left(\frac{\varepsilon}{4}\Delta - U^*U\right) + 2\left(\frac{\varepsilon}{4}\Delta - W^*W \right)\\
& = (2-u^*-w)^*(2-u^*-w) \\
&  \quad + 2\sum_i \mu_i^*\mu_i + 2 \sum_i \nu_i^*\nu_i + 2\left(\frac{\varepsilon}{4} - \varepsilon_0\right) \Delta,
\end{align*}
i.e. $(2-a^* - a) \leqslant_{\lceil R/2\rceil}\varepsilon \Delta$, as stated.

Note that if none of the generators in $S$ is involutive, the start of the induction provides a better estimate:
\[
-(1-s)^*(1-s) \leqslant_1 \Delta \quad\text{and}\quad -(1-st)^*(1-st) \leqslant_1 2\Delta,
\]
and we can use $\varepsilon \geqslant \SOSconedistm{R}{-1}$.
\end{proof}

\begin{dfn}
We say that an element $x \in IG^h$ admits an \emph{$\varepsilon$-approximate sum of squares decomposition supported in $B_R$} if $x + \varepsilon \Delta \in \Sigma^2_R IG$.
\end{dfn}

\begin{cor}\label{cor:asos-lambdaDelta}
Let $y\in IG^h$ and suppose that $y-\lambda \Delta$ admits an $\varepsilon$\nobreakdash-\hspace{0pt}approximate sum of squares decomposition supported in $B_R$,
with $\lambda \geqslant \varepsilon$ and $R \geqslant 1$. 
 Then $y \in \Sigma^2_R IG$.
\end{cor}

\begin{proof}
Set $x = y-\lambda \Delta$.
By definition,
\[y - \left(\lambda- \varepsilon\right)\Delta = x + \varepsilon\Delta = \sum \mu_i^*\mu_i\]
where each $\mu_i$ is supported in $B_R$.
Therefore
 \[y = \sum \mu_i^*\mu_i + \left(\lambda- \epsilon\right)\Delta\]
belongs to $\Sigma^2_R IG$ as long as $\lambda \geqslant \varepsilon$.
\end{proof}

\begin{lem}\label{lem: asos}
Let $ x\in IG^h$.
Suppose that $b \leqslant_R x$ for some $b$ supported in $B_{2R}$.
Then
$x$ admits an $\varepsilon$-approximate sum of squares decomposition supported in $B_R$ for every
$\varepsilon \geqslant \SOSconedist{R}\|b\|_1$.
\end{lem}

\begin{proof}
The lemma is a specialisation of \cite{Netzer2015}*{Lemma 2.1} which tracks the exact supports of constructed sum of squares decompositions.


Recall that $b \leqslant_R x$ means that
 \[x - b = \sum \xi_i^* \xi_i \in \Sigma^2_{R} IG .\]
We necessarily have $b \in IG^h$, and so we may write it as
\[b =
\sum_{g\in \supp{b}} b_g g =
\sum_{1\neq g\in \supp{b}}
\frac{-b_g}{2}(2 - g^* - g),\]
where $\supp{b}\subseteq B_{2R}$.

If some coefficient $b_g$ is negative, then we may include $\frac{|b_g|}{2}(2-g ^*-g)$ in the sum of squares decomposition $\Sigma \xi_i^*\xi_i$ on the right hand side.
This operation does not alter $R$, the radius of sum of squares decomposition.
Indeed, we may write $g = uw$ with both $u$ and $w$ in $B_R$. We then have  $(1-g)^*(1-g) = (u^*-w)^*(u^*-w)$.
Therefore, we may assume that all coefficients $b_g$ of $b$ are positive for $g \neq 1$.

By the previous lemma, 
for every non-trivial $g \in B_{2R}$ we have
\[
  \frac{b_g}{2}(1-g)^*(1-g) \leqslant_{R}\frac{b_g}{2}\delta\Delta,
\]
for every $\delta\geqslant\SOSconedist{2R}$.
Since every $b_g$ is positive (for $g \neq 1$), we have
\[
- b = \sum_{g \neq 1} \frac{b_g}{2}(1-g)^*(1-g) \leqslant_{R}
\sum_{g \neq 1} \frac{b_g}{2}\delta \Delta =
\frac{\Vert b \Vert_1}{4}\delta \Delta = \varepsilon \Delta,
\]
where we set $\delta = \frac {4 \varepsilon}{ \|b\|_1}$. Therefore $\varepsilon\Delta + b \in \Sigma^2_R IG$.
Finally,
\[x +\varepsilon\Delta = \sum_i \xi_i^*\xi_i + \left(b + \varepsilon\Delta \right)\]
provides the required sum of squares decomposition.

As in the previous lemma, we could use $\varepsilon \geqslant \SOSconedistm{R}{-1}\Vert b \Vert_1$ if none of the generators of $G$ were involutive.
\end{proof}

\section{Applications}

\subsection{Property (T) for \texorpdfstring{$\SL_n(\Z)$}{SLn(Z)}}

Let $\Delta_3$ denote the Laplacian of $\SL_3(\Z)$ with respect to the standard generating set.
It was shown in \cite{KalubaNowak2018} that
\[{\Delta_3}^2 - 0.2799\, \Delta_3 - b \in \Sigma^2_2 I\SL_3(\Z),\]
with remainder $b$ supported in $B_4$, such that $\|b\|_1 < 10^{-9}$.
This constitutes a $\left(4\cdot 10^{-9}\right)$-approximate sum of squares decomposition supported in $B_2$ of $x = {\Delta_3}^2 - 0.2799\, \Delta_3$
by \cref{lem: asos}, which yields the following proposition by \cref{cor:asos-lambdaDelta}.

\begin{prop}[\cite{KalubaNowak2018}]
\label{sl3}
The group $\SL_3(\Z)$ has Kazhdan's property~(T).
Moreover, endowed with the standard generating set, its Kazhdan radius is at most $2$.
\end{prop}

\begin{prop}
\label{adj3}
The element $\Adj_3 \in \R\SL_3(\Z)$ satisfies
\[
\Adj_3 - 0.157999 \Delta_3 \in \Sigma^2_2 I\SL_3(\Z).
\]
\end{prop}

\begin{proof}
A computer calculation yielded  elements $\xi_i \in I \SL_3(\Z)$ supported on $B_2$ such that
\[\Adj_3 - 0.158\Delta_3 - b = \sum\xi_i^*\xi_i \in \Sigma^2_2 I\SL_3(\Z).\]
The remainder $b$ is supported in $B_4$, and is of magnitude
$\|b\|_1 < 2\cdot 10^{-8}$.
Thus, $\Adj_3 - 0.158\Delta_3$ admits an $\left(8\cdot 10^{-8}\right)$-approximate sum of squares decomposition supported on $B_2$
and therefore $\Adj_3 - 0.157999 \Delta_3 \in \Sigma^2_2 I\SL_3(\Z)$ by \cref{cor:asos-lambdaDelta}.
\end{proof}

\begin{thm}
For every $n\geqslant 3$, the group $\SL_n(\Z)$ has Kazhdan's property~(T).
Moreover, endowed with the standard generating set, its Kazhdan radius is at most $2$.
\end{thm}
\begin{proof}
 This follows immediately by applying \cref{Adj+kOp>=0} (Method I, with $k=0$) and  \cref{adj3}.
\end{proof}

This constitutes a new proof of Kazhdan's property~(T) for $\SL_n(\Z)$ with $n \geqslant 3$.

\begin{rmk}(Kazhdan constants)\label{rmk: Explicit Kazhdan constants}
An explicit construction of vectors in representations of the groups $\SL_n(\Z)$ provided by A. \.{Z}uk (see \cite{Shalom1999}) established an upper bound on the Kazhdan constant.
A lower bound was obtained by M.~Kassabov \cite{Kassabov2005} via bounded generation property. To our knowledge these are the best known bounds on the Kazhdan constant for $\SL_n(\Z)$:
\[\frac{1}{42\sqrt{n}+860} \leqslant \kappa(\SL_n(Z), S_n) \leqslant \sqrt{\frac{2}{n}}.\]

By showing that $\Adj_3 - 0.157999 \Delta_3 \in \Sigma^2_2 I \SL_3(\Z)$ we obtain \[\sqrt{\frac{0.157999(n-2)}{n^2-n}} \leqslant \kappa(\SL_n(\Z), S_n)\]
for all $n\geqslant 3$.
\end{rmk}

\begin{rmk}
An even better estimate is obtained from the fact that
$\Adj_4 + \Op_4 - 0.82\Delta \in \Sigma^2_2 I SL_4(\Z).$
The details of this computation are provided in the jupyter notebook linked in \cref{replication}.
Using \cref{Adj+kOp>=0} leads to
\[\sqrt{\frac{0.41(n-2)}{n^2-n}} \leqslant \kappa(\SL_n(\Z), S_n)\]
for $n \geqslant 4$.

Similarly, a yet stronger estimate is implied by the existence of a sum of squares decomposition for
$\Adj_5 + 1.5\Op_5 - 1.5\Delta_5\in \Sigma^2_2 I SL_5(\Z)$. This is used to obtain (for $n\geqslant 6$)
\[\sqrt{\frac{0.5(n-2)}{n^2-n}} \leqslant \kappa(\SL_n(\Z), S_n).\]

\medskip
While asymptotically optimal, the previously known bounds differ in almost two orders of magnitude even for $n=1000$.
The lower bound obtained here is asymptotically a half of the upper bound and is already $\frac{2}{5}$-ths of the upper bound for $n = 6$.
\end{rmk}

\subsection{Property (T) for \texorpdfstring{$\Aut(F_n)$}{Aut(Fn)}}
\label{section (T) for AutFn}

Let $\Delta_5$ denote the Laplacian of $\SAut(F_5)$ with respect to the standard generating set.
It was shown in \cite{Kalubaetal2017} (see \cite{Kaluba2017a}) that
\[{\Delta_5}^2 - 1.3\, \Delta_5 - b \in \Sigma^2_2 I\SAut(F_5)\]
with remainder $b$ supported in $B_4$ and such that $\|b\|_1 < 3\cdot 10^{-6}$.
As in the case of $\SL_3(\Z)$ this yields the following proposition.

\begin{thm}[\cite{Kalubaetal2017}]
\label{saut5}
The group $\SAut(F_5)$ has Kazhdan's property~(T).
Moreover, endowed with the standard generating set, its Kazhdan radius is at most $2$.
\end{thm}

\begin{prop}
\label{Adj5+2Op5}
The element $\Adj_5 + 2 \Op_5$ in $\R\SAut(F_5)$ satisfies
\[
\Adj_5 + 2 \Op_5 -  0.278 \Delta_5 \in \Sigma^2_2 I\SAut(F_5).
\]
\end{prop}
\begin{proof}
Computer calculations for $y = \Adj_5 + 2 \Op_5 \in \R \SAut(F_5)$ found elements $\xi_i$ such that
\[y - 0.28 \Delta_5 - b = \sum \xi_i^* \xi_i \in \Sigma^2_2 I\SAut(F_5)\]
(for replication details see \cref{replication}).
The remainder $b$ is supported in $B_4$ and $\Vert b \Vert_1 < 0.0009423$.
Thus $ y - 0.28\Delta_5$ admits a $\left(0.00189\right)$-approximate sum of squares decomposition supported in $B_2$,
and so by \cref{cor:asos-lambdaDelta} $\Adj_5 + 2 \Op_5 -  0.278 \Delta_5 \in \Sigma^2_2 I\SAut(F_5)$ .
\end{proof}

\begin{rmk}
\label{Adj5+3Op5}
Additional computations (see \cref{replication}) show that
\[
\Adj_5 + 3 \Op_5 -  1.37 \Delta_5 \in \Sigma^2_2 I\SAut(F_5).
\]
\end{rmk}

\begin{thm}
\label{main thm long}
 For every $n\geqslant 6$, the groups $\SAut(F_n)$, $\Aut(F_n)$ and $\mathrm{Out}(F_n)$ have Kazhdan's property~(T).
 Moreover, endowed with the standard generating set, the Kazhdan radius of $\SAut(F_n)$ is at most $2$.
\end{thm}
\begin{proof}
For $\SAut(F_n)$, the result follows immediately from \cref{Adj+kOp>=0} (Method I) together with \cref{Adj5+2Op5} for $n\geqslant 7$ (setting $k=2$), and from \cref{Adj+kOp>=0 Xn} (Method  II) together with \cref{saut5} for $n=6$ (setting $k=3$).

For $\Aut(F_n)$ and $\operatorname{Out}(F_n)$ the result follows from the fact that property~(T) ascends to finite index overgroups and descends to quotients.
\end{proof}

\hyphenation{met-abel-ian}
\begin{rmk}[Suggested by M. Mimura]
Let $\Phi_n$ denote the free metabelian group of rank $n$.
It follows from \cref{main thm long} that $\Aut(\Phi_n)$ has property (T) for $n\ge 5$,
since $\Aut(\Phi_n)$ is a quotient of $\Aut(F_n)$ by a result of Bachmuth and Mochizuki~\cite{BachmuthMochizuki1985} (which holds also for $n=4$).
It was also shown by the same authors in~\cite{BachmuthMochizuki1982} that $\Aut(\Phi_3)$ is not the image of $\Aut(F_3)$ -- in fact, they showed that $\Aut(\Phi_3)$ is not finitely generated.
\end{rmk}

\begin{rmk}(Kazhdan constants)
These results lead to the following lower bounds of the Kazhdan constant $\kappa_n = \kappa(\SAut(F_n), S_n)$. We have $\lambda_5 \geqslant 1.29999$ and $\mu \geqslant 0.277$, and hence
\begin{align*}
0.18027
  \leqslant \sqrt{\frac{2\lambda_5}{4(5^2-5)}} &
    \leqslant \kappa_5,\quad \text{by \cite{Kalubaetal2017};}\\
0.00983
  \leqslant \sqrt{\frac{3\lambda_5 + \mu}{6h_6(6^2-6)}} &
    \leqslant \kappa_6,\quad \text{by \cref{Adj+kOp>=0 Xn Kaz};}\\
0.07426
  \leqslant \sqrt{\frac{0.278\cdot (7-2)}{6\cdot (7^2-7)}} &
    \leqslant \kappa_7,\\
0.07045
  \leqslant \sqrt{\frac{0.278\cdot(8-2)}{6\cdot (8^2-8)}} &
    \leqslant \kappa_8,\quad \text{by \cref{Adj+kOp>=0,Adj5+2Op5};}\\
0.14899
  \leqslant \sqrt{\frac{1.37\cdot(9-2)}{6\cdot (9^2-9)}} &
    \leqslant \kappa_9,\\
\intertext{and in general, for $n\geqslant 9$, by \cref{Adj+kOp>=0,Adj5+3Op5} }
  \sqrt{\frac{1.37(n-2)}{6(n^2-n)}} &\leqslant \kappa_n.
\end{align*}

One could potentially improve those bounds for $n=6$ by obtaining a sum of squares decompositions for $\Adj_5 + 1.5 \Op_5 - \lambda\Delta_5$ and using \cref{Adj+kOp>=0} (Method I), however the authors were not able to do so.

\end{rmk}

\subsubsection*{Product Replacement Algorithm}

Product replacement algorithm is an algorithm which utilizes a random walk to generate \emph{pseudorandom} elements in finite groups.
In the classical setting of \cite{Celleretal1995}, given a (huge, but finite) group $G$ generated by $n$ elements one considers a random walk on $4n(n-1)$-regular graph $\Gamma$:
\begin{itemize}
\item the vertices correspond to $n$-generating tuples of $G$ and
\item an edge connects two such tuples if one can be obtained by applying one of the standard generators of $\SAut(F_n)$ to the other.
\end{itemize}
Using the notation of \cite{LubotzkyPak2001}*{Theorem 3.1} (and noting that the action of $\Z_2 \wr \operatorname{Sym}_n$ is transitive on the generating set $S_n$) the provable convergence rate of the Product Replacement random walk  can be estimated as
\[\left\|Q_{(g_1, \ldots, g_n)}^t - U \right\|_{\operatorname{tv}} \leqslant \varepsilon \quad \text{for}\quad t\geqslant \frac{24(n^2-n)}{1.37(n-2)}\log\left(\frac{|\Gamma|}{\varepsilon}\right),\]
where $\log(|\Gamma|) \leqslant n\log|G|$. Thus, for a given group $G$ with an $n$-generating tuple the mixing time is logarithmic in $|G|$ and quadratic in $n$.

Product Replacement Algorithm is implemented e.g. in GAP \cite{GAP}, however,
the applicability of the obtained bounds to actual implementations is debatable:
the GAP system implements a version of PRA with an \emph{accumulator} and it is not clear whether the results of \cite{LubotzkyPak2001} apply to it.
The implementations of other computer algebra systems are not available to be analyzed, so no conclusion about them can be drawn.

\appendix
\section{Replication of the results}
\label{replication}

The algorithm computing approximate sum of squares decompositions has not changed since its description in \cite{Kalubaetal2017},
however the architecture of the code has had, to allow for defining semi-definite programs for sum of squares decompositions for arbitrary group ring elements and order units.
The source code is written in the Julia programming language \cite{Julia2017}. It uses the \verb`SCS` solver \cite{O'Donoghue2016} through the \verb`JuMP` modeling language \cite{JuMP2017}, as well as \verb`AbstractAlgebra` package \cite{Nemo2017} as dependencies.

For an example on how to run and experiment with the code we refer the reader to the provided \noindent\href{https://nbviewer.jupyter.org/gist/kalmarek/03510181bc1e7c98615e86e1ec580b2a}{jupyter notebook}.
The dataset and the detailed replication instructions are deposited in the Zenodo research data repository \cite{Kaluba2018}. This repository contains both the jupyter notebook linked above, dealing with the case of special linear groups, as well as a second jupyter notebook used for the case of automorphism groups of free groups.

\paragraph{Replication}
The first jupyter notebook can be used to perform
the described computations for $\SL_n(\Z)$ ($n=3,4,5$).
Carrying out the calculations should take no longer than $10$ minutes on a regular laptop computer.

As an interesting detail we note that the optimization problem for $\Adj_4 +\Op_4$ is much better conditioned than the one for $\Adj_3$
e.g. exhibits faster convergence and better accuracy with relatively larger $\lambda$,
despite being larger in size (i.e. number of variables and constraints).
The same is even more true for $\Adj_5 + 1.5\Op_5$.

\medskip
We provide a precomputed solution for the more computationally demanding case of $\SAut(F_5)$. Using this solution, the second jupyter notebook provides verification tools for the existence of approximate sum of squares decompositions of the elements
\begin{align*}
\Adj_5 + 2\Op_5 - 0.28\Delta \in I\SAut(F_5), &\text{ and}\\
\Adj_5 + 3\Op_5 - 1.4\Delta \in I\SAut(F_5).
\end{align*}
The precomupted data includes the elements
$\Delta_5 , \Adj_5 , \Op_5 $
as well as the orbit decomposition, multiplication table and matrices $Q$, which provide the approximate sum of squares decomposition.

The second jupyter notebook also contains  code which may be used to find an approximate sum of squares decomposition of
\[
\Adj_4 + 100 \Op_4 - 0.1\Delta \in I\SAut(F_4).
\]
The reason for including this computation is that it performs all the steps without any precomputed data. The running time should be less than 10 minutes. Once the existence of a sum of squares is certified, \cref{method I} with $k=100$ proves that $\SAut(F_n)$ has property~(T) for every $n\geqslant 103$.

\bibliography{FAb}

\begin{bibdiv}
\begin{biblist}

\bib{Bekkaetal2008}{book}{
      author={Bekka, Bachir},
      author={de~la Harpe, Pierre},
      author={Valette, Alain},
       title={Kazhdan's property ({T})},
      series={New Mathematical Monographs},
   publisher={Cambridge University Press, Cambridge},
        date={2008},
      volume={11},
        ISBN={978-0-521-88720-5},
         url={https://doi.org/10.1017/CBO9780511542749},
      review={\MR{2415834}},
}

\bib{Julia2017}{article}{
      author={Bezanson, Jeff},
      author={Edelman, Alan},
      author={Karpinski, Stefan},
      author={Shah, Viral~B.},
       title={{Julia: A Fresh Approach to Numerical Computing}},
        date={2017jan},
        ISSN={0036-1445},
     journal={SIAM Review},
      volume={59},
      number={1},
       pages={65\ndash 98},
}

\bib{BachmuthMochizuki1982}{article}{
      author={Bachmuth, S.},
      author={Mochizuki, H.~Y.},
       title={The nonfinite generation of {${\rm Aut}(G)$}, {$G$} free
  metabelian of rank {$3$}},
        date={1982},
        ISSN={0002-9947},
     journal={Trans. Amer. Math. Soc.},
      volume={270},
      number={2},
       pages={693\ndash 700},
         url={https://doi.org/10.2307/1999870},
      review={\MR{645339}},
}

\bib{BachmuthMochizuki1985}{article}{
      author={Bachmuth, Seymour},
      author={Mochizuki, Horace~Y.},
       title={{${\rm Aut}(F)\to{\rm Aut}(F/F'')$} is surjective for free group
  {$F$} of rank {$\geq 4$}},
        date={1985},
        ISSN={0002-9947},
     journal={Trans. Amer. Math. Soc.},
      volume={292},
      number={1},
       pages={81\ndash 101},
         url={https://doi.org/10.2307/2000171},
      review={\MR{805954}},
}

\bib{Breuillard2014}{incollection}{
      author={Breuillard, Emmanuel},
       title={Expander graphs, property {$(\tau)$} and approximate groups},
        date={2014},
   booktitle={Geometric group theory},
      series={IAS/Park City Math. Ser.},
      volume={21},
   publisher={Amer. Math. Soc., Providence, RI},
       pages={325\ndash 377},
      review={\MR{3329732}},
}

\bib{BridsonVogtmann2006}{incollection}{
      author={Bridson, Martin~R.},
      author={Vogtmann, Karen},
       title={Automorphism groups of free groups, surface groups and free
  abelian groups},
        date={2006},
   booktitle={Problems on mapping class groups and related topics},
      series={Proc. Sympos. Pure Math.},
      volume={74},
   publisher={Amer. Math. Soc., Providence, RI},
       pages={301\ndash 316},
         url={https://doi.org/10.1090/pspum/074/2264548},
      review={\MR{2264548}},
}

\bib{BogopolskiVikentiev2010}{incollection}{
      author={Bogopolski, O.},
      author={Vikentiev, R.},
       title={Subgroups of small index in {${\rm Aut}(F_n)$} and {K}azhdan's
  property ({T})},
        date={2010},
   booktitle={Combinatorial and geometric group theory},
      series={Trends Math.},
   publisher={Birkh\"auser/Springer Basel AG, Basel},
       pages={1\ndash 17},
      review={\MR{2744013}},
}

\bib{Celleretal1995}{article}{
      author={Celler, Frank},
      author={Leedham-Green, Charles~R.},
      author={Murray, Scott~H.},
      author={Niemeyer, Alice~C.},
      author={O'brien, E.A.},
       title={Generating random elements of a finite group},
        date={1995},
     journal={Communications in Algebra},
      volume={23},
      number={13},
       pages={4931\ndash 4948},
      eprint={https://doi.org/10.1080/00927879508825509},
         url={https://doi.org/10.1080/00927879508825509},
}

\bib{JuMP2017}{article}{
      author={Dunning, Iain},
      author={Huchette, Joey},
      author={Lubin, Miles},
       title={Jump: A modeling language for mathematical optimization},
        date={2017},
     journal={SIAM Review},
      volume={59},
      number={2},
       pages={295\ndash 320},
}

\bib{Nemo2017}{inproceedings}{
      author={Fieker, Claus},
      author={Hart, William},
      author={Hofmann, Tommy},
      author={Johansson, Fredrik},
       title={{Nemo/Hecke: Computer Algebra and Number Theory Packages for the
  Julia Programming Language}},
        date={2017},
   booktitle={Proceedings of the 2017 acm on international symposium on
  symbolic and algebraic computation},
      series={ISSAC '17},
   publisher={ACM},
     address={New York, NY, USA},
       pages={157\ndash 164},
}

\bib{Fisher-IMRN}{article}{
      author={Fisher, David},
       title={{${\rm Out}(F_n)$} and the spectral gap conjecture},
        date={2006},
        ISSN={1073-7928},
     journal={Int. Math. Res. Not.},
       pages={Art. ID 26028, 9},
      review={\MR{2250018}},
}

\bib{Fujiwara2017}{article}{
      author={Fujiwara, Koji},
      author={Kabaya, Yuichi},
       title={Computing kazhdan constants by semidefinite programming},
        date={2017},
     journal={Experimental Mathematics},
      volume={0},
      number={0},
       pages={1\ndash 12},
      eprint={https://doi.org/10.1080/10586458.2017.1396509},
         url={https://doi.org/10.1080/10586458.2017.1396509},
}

\bib{Gilman1977}{article}{
      author={Gilman, Robert},
       title={Finite quotients of the automorphism group of a free group},
        date={1977},
        ISSN={0008-414X},
     journal={Canad. J. Math.},
      volume={29},
      number={3},
       pages={541\ndash 551},
         url={http://dx.doi.org/10.4153/CJM-1977-056-3},
      review={\MR{0435226}},
}

\bib{GrunewaldLubotzky2006}{article}{
      author={Grunewald, F.},
      author={Lubotzky, A.},
       title={{Linear Representations of the Automorphism Group of a Free
  Group}},
        date={2006-06},
     journal={Geometric and Functional Analysis},
      volume={18},
      number={5},
       pages={1564\ndash 1608},
      eprint={arXiv:math/0606182},
}

\bib{Kassabov2005}{article}{
      author={Kassabov, Martin},
       title={Kazhdan constants for {${\rm SL}_n({\mathbb Z})$}},
        date={2005},
        ISSN={0218-1967},
     journal={Internat. J. Algebra Comput.},
      volume={15},
      number={5-6},
       pages={971\ndash 995},
         url={https://doi.org/10.1142/S0218196705002712},
      review={\MR{2197816}},
}

\bib{Kaluba2018}{misc}{
      author={Kaluba, Marek},
      author={Kielak, Dawid},
      author={Nowak, Piotr~W.},
       title={{Approximate sum of squares decompositions for
  $\operatorname{Adj}_5 + k\cdot\operatorname{Op}_5 - \lambda \Delta_5 \in
  I\operatorname{SAut}(F_5)$}},
        date={2020},
         url={https://zenodo.org/record/4118043},
}

\bib{KalubaNowak2018}{article}{
      author={Kaluba, Marek},
      author={Nowak, Piotr~W.},
       title={Certifying numerical estimates of spectral gaps},
        date={2018},
        ISSN={1867-1144},
     journal={Groups Complex. Cryptol.},
      volume={10},
      number={1},
       pages={33\ndash 41},
         url={https://doi.org/10.1515/gcc-2018-0004},
      review={\MR{3794918}},
}

\bib{Kaluba2017a}{misc}{
      author={Kaluba, Marek},
      author={Nowak, Piotr~W.},
      author={Ozawa, Narutaka},
       title={{An approximation of the spectral gap for the Laplace operator on
  $\operatorname{SAut}(\mathbb{F}_5)$}},
        date={2017},
         url={https://zenodo.org/record/1913734},
}

\bib{Kalubaetal2017}{article}{
      author={Kaluba, Marek},
      author={Nowak, Piotr~W.},
      author={Ozawa, Narutaka},
       title={{${\rm Aut}(\mathbb F_5)$} has property {$(T)$}},
        date={2019},
        ISSN={0025-5831},
     journal={Math. Ann.},
      volume={375},
      number={3-4},
       pages={1169\ndash 1191},
         url={https://doi.org/10.1007/s00208-019-01874-9},
      review={\MR{4023374}},
}

\bib{LubotzkyPak2001}{article}{
      author={Lubotzky, Alexander},
      author={Pak, Igor},
       title={The product replacement algorithm and {K}azhdan's property
  ({T})},
        date={2001},
        ISSN={0894-0347},
     journal={J. Amer. Math. Soc.},
      volume={14},
      number={2},
       pages={347\ndash 363},
         url={https://doi.org/10.1090/S0894-0347-00-00356-8},
      review={\MR{1815215}},
}

\bib{McCool1989}{article}{
      author={McCool, James},
       title={A faithful polynomial representation of {${\rm Out}\,F_3$}},
        date={1989},
        ISSN={0305-0041},
     journal={Math. Proc. Cambridge Philos. Soc.},
      volume={106},
      number={2},
       pages={207\ndash 213},
         url={https://doi.org/10.1017/S0305004100078026},
      review={\MR{1002533}},
}

\bib{Netzer2015}{article}{
      author={Netzer, Tim},
      author={Thom, Andreas},
       title={{Kazhdan's Property (T) via Semidefinite Optimization}},
        date={2015jul},
        ISSN={1058-6458},
     journal={Experimental Mathematics},
      volume={24},
      number={3},
       pages={371\ndash 374},
  url={http://www.tandfonline.com/doi/full/10.1080/10586458.2014.999149},
}

\bib{O'Donoghue2016}{article}{
      author={O'Donoghue, Brendan},
      author={Chu, Eric},
      author={Parikh, Neal},
      author={Boyd, Stephen},
       title={Conic optimization via operator splitting and homogeneous
  self-dual embedding},
        date={2016Jun},
        ISSN={1573-2878},
     journal={Journal of Optimization Theory and Applications},
      volume={169},
      number={3},
       pages={1042\ndash 1068},
}

\bib{Ozawa2016}{article}{
      author={Ozawa, Narutaka},
       title={Noncommutative real algebraic geometry of {K}azhdan's property
  ({T})},
        date={2016},
        ISSN={1474-7480},
     journal={J. Inst. Math. Jussieu},
      volume={15},
      number={1},
       pages={85\ndash 90},
         url={https://doi.org/10.1017/S1474748014000309},
      review={\MR{3427595}},
}

\bib{Shalom1999}{article}{
      author={Shalom, Yehuda},
       title={Bounded generation and {K}azhdan's property ({T})},
        date={1999},
        ISSN={0073-8301},
     journal={Inst. Hautes \'{E}tudes Sci. Publ. Math.},
      number={90},
       pages={145\ndash 168 (2001)},
         url={http://www.numdam.org/item?id=PMIHES_1999__90__145_0},
      review={\MR{1813225}},
}

\bib{GAP}{manual}{
      author={{The GAP Group}},
       title={{GAP -- Groups, Algorithms, and Programming, Version 4.7.8}},
organization={The GAP~Group},
        date={2015},
         url={\verb+(http://www.gap-system.org)+},
        note={\texttt{http://www.gap-system.org}},
}

\end{biblist}
\end{bibdiv}

\end{document}